\documentclass[10pt, a4paper]{amsart}

\pdfoutput=1

\usepackage{mathrsfs, amsmath, amsthm, amssymb, stmaryrd, enumerate}%
\usepackage{tikz}%
\usepackage[linktocpage]{hyperref}%
\usepackage[all]{xy}%
\xyoption{2cell}%
\UseAllTwocells%

\newdir{t>}{{\UseTips\dir{>}}}

\DeclareMathOperator{\id}{id}
\DeclareMathOperator{\el}{el}
\DeclareMathOperator{\op}{op}

\DeclareMathOperator{\Mod}{\mathbf{Mod}}
\DeclareMathOperator{\Rep}{Rep}
\DeclareMathOperator{\Vect}{\mathbf{Vect}}

\DeclareMathOperator{\fp}{fp}
\DeclareMathOperator{\fd}{fd}
\DeclareMathOperator{\Hom}{Hom}
\DeclareMathOperator{\End}{End}

\DeclareMathOperator{\Sym}{Sym}
\DeclareMathOperator{\sgn}{sgn}

\DeclareMathOperator{\Rex}{\mathbf{Rex}}

\DeclareMathOperator{\Ind}{Ind}

\DeclareMathOperator{\Coh}{\mathbf{Coh}}
\DeclareMathOperator{\QCoh}{\mathbf{QCoh}}

\DeclareMathOperator{\fpqc}{\mathit{fpqc}}

\DeclareMathOperator{\pr}{pr}

\DeclareMathOperator{\Lex}{\mathbf{Lex}}

\newcommand{\ca}[1]{\mathscr{#1}}

\newcommand{\ten}[1]{\mathop{{\otimes}_{#1}}}

\newcommand{\defl}{\mathrel{\mathop:}=}

\theoremstyle{plain}
\newtheorem{thm}{Theorem}[subsection]
\newtheorem{prop}[thm]{Proposition}
\newtheorem{lemma}[thm]{Lemma}
\newtheorem{cor}[thm]{Corollary}

\theoremstyle{definition}

\newtheorem{rmk}[thm]{Remark}
\newtheorem{dfn}[thm]{Definition}

\newtheoremstyle{citing}{}{}{\itshape}{}{\bfseries}{.}{ }{\thmnote{#3}}
\theoremstyle{citing}

\newtheoremstyle{citingdfn}{}{}{}{}{\bfseries}{.}{ }{\thmnote{#3}}
\theoremstyle{citingdfn}

\numberwithin{equation}{section}


\author{Daniel Sch\"appi}
\thanks{This research was supported by the DFG grant: SFB 1085 ``Higher invariants.''}
\address{Fakult{\"a}t f{\"u}r Mathematik,
Universit{\"a}t Regensburg,
93040 Regensburg,
Germany}
\email{daniel.schaeppi@ur.de}

\title{Uniqueness of fiber functors and universal Tannakian categories}

\begin{document}

\begin{abstract}
 The principal aim of this note is to give an elementary proof of the fact that any two fiber functors of a Tannakian category are locally isomorphic. This builds on an idea of Deligne concerning scalar extensions of Tannakian categories and implements a proof strategy which Deligne attributes to Grothendieck. Besides categorical generalities, the proof merely relies on basic properties of exterior powers and the classification of finitely generated modules over a principal ideal domain.
 
 Using related ideas (but less elementary means) we also present an alternative characterization of Tannakian categories among the more general weakly Tannakian categories. As an application of this result we can construct from any right exact symmetric monoidal abelian category with simple unit object a universal Tannakian category associated to it.
\end{abstract}




\maketitle

\tableofcontents

\section{Introduction}\label{section:introduction}

 The notion of a Tannakian category provides a characterization of categories of coherent sheaves $\Coh(G)$ of an $\fpqc$-gerbe $G$ with affine band and affine diagonal. It is comparatively straightforward to show from the definition that a Tannakian category is equivalent to the category of coherent sheaves $\Coh(X)$ of some $\fpqc$-stack $X$, but the proof that the stack in question is in fact a gerbe presented in \cite{DELIGNE} is quite technical, in particular in the case where the ground field is imperfect.
 
 One needs to show is that any two fiber functors of a ($k$-linear) Tannakian category $\ca{T}$ are locally isomorphic in the $\fpqc$-topology, or, equivalently, that the diagonal of the corresponding stack is faithfully flat. This, in turn, follows readily from the fact that the Deligne tensor product $\ca{T} \boxtimes \ca{T}$ of $\ca{T}$ with itself is again a Tannakian category. In \cite{DELIGNE_SEMI_SIMPLE}, Deligne mentions an alternative strategy to prove this result suggested by Grothendieck, which relies on a scalar extension result for Tannakian categories. Namely, starting from a fiber functor $w \colon \ca{T} \rightarrow \Vect_K^{\fd}$, one obtains a faithful and exact tensor functor
 \[
 w \boxtimes \id \colon \ca{T} \boxtimes \ca{T} \rightarrow \Vect_K^{\fd} \boxtimes \ca{T}
 \]
 with target the scalar extension of $\ca{T}$. Since this faitfhul and exact tensor functor detects simple objects and objects with duals, it therefore suffices to check that $\Vect_K^{\fd} \boxtimes \ca{T}$ is a Tannakian category over $K$.
 
 In \cite[\S 5]{DELIGNE_SEMI_SIMPLE}, Deligne gives an elementary proof of this scalar extension result in the case where $K$ is a \emph{finite} field extension of $k$. Note that this result can already be used as outlined above if there exists a fiber functor $\ca{T} \rightarrow \Vect_K^{\fd}$ where $K$ is a finite extension of $k$. Moreover, we know that such a fiber functor always exists in the case where $\ca{T}$ is finitely generated by \cite[Corollaire~6.20]{DELIGNE}. However, the proof of this Corollary involves precisely the most technical parts of \cite{DELIGNE} which we want to avoid.
 
 In \S \ref{section:background}, we provide the necessary background on weakly Tannakian categories and scalar extensions of linear categories. In \S \ref{section:scalar_extension}, we recall some basic properties of exterior powers and complete the proof that the scalar extension of a Tannakian category is again Tannakian. In \S \ref{section:universal_tannakian_categories}, we give a new characterization of Tannakian categories among the weakly Tannakian categories. As an application, we study the question when the universal weakly Tannakian categories introduced in \cite[\S 5.6]{SCHAEPPI_COLIMITS} are in fact Tannakian. Namely, we show that the universal weakly Tannakian category associated to a right exact symmetric monoidal abelian category is Tannakian if the unit object of the original category is simple.
 
 \section*{Acknowledgements}
 This work was carried out at the SFB 1085 ``Higher invariants'' in Regensburg whose support is gratefully acknowledged.
 
\section{Background}\label{section:background}

\subsection{Weakly Tannakian categories}\label{section:weakly_tannakian}

 As mentioned in the introduction, the central aim of this note is to give an elementary proof of the fact that any two fiber functors of a Tannakian category are locally isomorphic. We start by explaining in more detail how this follows from the fact that the Deligne tensor product of two Tannakian categories is again Tannakian. As it turns out, the corresponding result for tensor products of \emph{weakly} Tannakian categories is simpler to prove. At the same time it provides a conceptual way to see this relationship. 
 
 Throughout this section, we fix a field $k$ (though most of what we say is true for an arbitrary commutative base ring). We call a finitely cocomplete $k$-linear category $\ca{A}$ \emph{ind-abelian} if $\Ind(\ca{A})$ is an abelian category. See also \cite[\S 2]{SCHAEPPI_INDABELIAN} for an alternative characterization of ind-abelian categories.
 
 \begin{dfn}\label{dfn:weakly_tannakian}
  Let $\ca{A}$ be an ind-abelian $k$-linear symmetric monoidal category whose tensor product is right exact in each variable. Then $\ca{A}$ is called \emph{weakly Tannakian} if the following hold:
  \begin{enumerate}
  \item[(i)] There exists a \emph{fiber functor} $w \colon \ca{A} \rightarrow \Mod^{\fp}_A$ for some commutative $k$-algebra $A$, that is, a faithful right exact symmetric strong monoidal functor whose extension to ind-objects is exact.
  \item[(ii)] The category $\ca{A}$ is generated by objects with duals: for every object $A \in \ca{A}$, there exists an object $A^{\prime} \in \ca{A}$ with a dual and an epimorphism $A^{\prime} \rightarrow A$.
  \end{enumerate}
 \end{dfn}

 Clearly every Tannakian category is weakly Tannakian. Other examples include the category of finitely presentable modules over a commutative $k$-algebra and the category of coherent sheaves of a noetherian quasi-projective scheme (here the object $A^{\prime}$ in (ii) above can even be taken to be a finite direct sum of line bundles). Any $\fpqc$-covering of the scheme by an affine scheme yields a fiber functor by taking the pullback of a coherent sheaf to the cover. In fact, this example can be suitably generalized to include \emph{all} weakly Tannakian categories. This is made precise in the following theorem. We call a quasi-compact semi-separated $\fpqc$-stack an \emph{Adams stack} if the vector bundles form a generator of its category of quasi-coherent sheaves.
 
 \begin{thm}\label{thm:tannaka_duality}
 The pseudofunctor which sends an Adams-stack $X$ to the symmetric monoidal category $\QCoh_{\fp}(X)$ of finitely presentable quasi-coherent sheaves on $X$ gives a contravariant biequivalence between the 2-category of Adams stacks and the 2-category of weakly Tannakian categories and right exact symmetric strong monoidal functors between them.
 \end{thm}

\begin{proof}
 This is proved in \cite[Theorem~1.6]{SCHAEPPI_INDABELIAN} and \cite[Theorem~1.3.3]{SCHAEPPI_STACKS}.
\end{proof} 
 
 The central ingredient of the proof of this theorem is the Beck (co)monadicity theorem, applied to the functor $\Ind(w) \colon \Ind(\ca{A}) \rightarrow \Mod_A$. One needs to observe that the induced comonad is given by a Hopf algebroid (the dual notion of an affine groupoid), which relies on the two facts that tensor functors preserve duals and that modules with duals are precisely the finitely generated projective ones, whence homming out of them preserves colimits.

 Recall that the \emph{Deligne tensor product} $\ca{A} \boxtimes \ca{B}$ of two abelian categories $\ca{A}$ and $\ca{B}$ is the universal abelian category with a $k$-linear functor
 \[
  \ca{A} \times \ca{B} \rightarrow \ca{A} \boxtimes \ca{B}
 \]
 of two variables which is right exact in each variable. This tensor product does unfortunately not always exist. However, if instead of abelian categories we merely consider additive categories with cokernels (equivalently, additive categories with finite colimits), we can still talk about right exact functors, and in this world, the corresponding tensor product always exists (see \cite[\S 6.5]{KELLY_BASIC}). We call the latter tensor product the \emph{Kelly tensor product}. It turns out that the Deligne tensor product of two abelian categories exists if and only if the Kelly tensor product happens to be an abelian category, so the Kelly tensor product is really a generalization the Deligne tensor product. The key result about weakly Tannakian categories that we will need is the following theorem.
 
 \begin{thm}\label{thm:kelly_tensor_weakly_tannakian}
 If $\ca{A}$ and $\ca{B}$ are weakly Tannakian categories with fiber functors $v \colon \ca{A} \rightarrow \Mod_A$ and $w \colon \ca{B} \rightarrow \Mod_B$, then the Kelly tensor product $\ca{A} \boxtimes \ca{B}$ is weakly Tannakian and
 \[
 v \boxtimes w \colon \ca{A} \boxtimes \ca{B} \rightarrow \Mod_A^{\fp} \boxtimes \Mod_B^{\fp} \simeq \Mod_{A\otimes B}^{\fp}
 \]
 is a fiber functor.
 \end{thm}
\begin{proof}
 This is proved in \cite[Theorem~6.6]{SCHAEPPI_INDABELIAN}.
\end{proof}
 
 If we combine these two theorems, we obtain the following important corollary.
 
\begin{cor}\label{cor:coherent_sheaves_on_products}
 For any two Adams stacks $X$ and $Y$, there is a canonical equivalence
 \[
 \QCoh_{\fp}(X \times Y) \simeq \QCoh_{\fp}(X) \boxtimes \QCoh_{\fp}(Y)
 \]
 of $k$-linear symmetric monoidal categories. In particular, $X$ has a faithfully flat diagonal if and only if the functor
 \[
 \QCoh_{\fp}(X) \boxtimes \QCoh_{\fp}(X) \rightarrow \QCoh_{\fp}(X)
 \]
 induced by the tensor product of quasi-coherent sheaves is faithful and exact.
\end{cor}

\begin{proof}
 A product of Adams stacks is again an Adams stacks, and the Kelly tensor product of two symmetric monoidal categories whose tensor product is right exact in each variable provides a coproduct among such categories. Combining this with Theorems~\ref{thm:tannaka_duality} and \ref{thm:kelly_tensor_weakly_tannakian}, we find that both sides represent the coproduct of weakly Tannakian categories, so they must be equivalent. A more detailed proof of this can be found in \cite[Theorem~1.7]{SCHAEPPI_INDABELIAN}.
 
 The second claim follows from the first since the functor induced by the tensor product corresponds to the diagonal of the stack under Tannaka duality.
\end{proof}

 This corollary shows in particular that an Adams stack $X$ is a gerbe if the category $\QCoh_{\fp}(X) \boxtimes \QCoh_{\fp}(X)$ is Tannakian. Indeed, it is well-known that any right exact symmetric strong monoidal $k$-linear functor whose domain is Tannakian and whose target is (weakly) Tannakian is automatically faithful and exact: such a stack $X$ is clearly non-empty and the evident cartesian square shows that any two points of a stack with faithfully flat diagonal are $\fpqc$-locally isomorphic. Another consequence of Theorem~\ref{thm:kelly_tensor_weakly_tannakian} is the desired reduction of the local uniqueness question to scalar extension. For later use, prove the following lemma separately.
 
\begin{lemma}\label{lemma:rigid_abelian}
 Any rigid ind-abelian category is abelian. In particular, a weakly Tannakian category is Tannakian if and only if it is rigid and the endomorphism ring of the unit object is a field.
\end{lemma}

\begin{proof}
 Any rigid category with cokernels also has kernels, and any ind-abelian category with kernels must be abelian since the embedding into ind-objects preserves kernels.
 
 If  $\ca{A}$ is also weakly Tannakian and the endomorphism ring of the unit is a field, then the unit object is simple, hence any right exact functor $\ca{A} \rightarrow \Vect_K$ is faithful and exact (see \cite[Corollaire~2.10]{DELIGNE}), so we get the desired fiber functor with values in $K$-vector spaces by composing the fiber functor $w \colon \ca{A} \rightarrow \Mod_A$ (which exists by Definition~\ref{dfn:weakly_tannakian}) with the tensor functor induced by any epimorphism $A \rightarrow K$. To see that such an epimorphism exists, we only need to observe that $A \neq 0$, which follows from the fact that $w$ is faithful.
\end{proof}
 
 \begin{cor}\label{cor:scalar_to_uniqueness}
 Let $\ca{T}$ and $\ca{T}^{\prime}$ be Tannakian categories over $k$ with fiber functors $w \colon \ca{T} \rightarrow \Vect_K^{\fd}$ and $w^{\prime} \colon \ca{T}^{\prime} \rightarrow \Vect_{K^{\prime}}^{\fd}$. If the scalar extension $\ca{T} \boxtimes \Vect^{\fd}_{K^{\prime}}$ of $\ca{T}$ from $k$ to $K^{\prime}$ is a Tannakian category over $K^{\prime}$, then the tensor product $\ca{T} \boxtimes \ca{T}^{\prime}$ is a Tannakian category over $k$. In particular, if the scalar extension of $\ca{T}$ from $k$ to $K$ is Tannakian over $K$, then any two fiber functors on $\ca{T}$ are $\fpqc$-locally isomorphic.
 \end{cor}
 
 \begin{proof}
 We can apply Theorem~\ref{thm:kelly_tensor_weakly_tannakian} to the pair $(\ca{T},w)$ and $(\ca{T}^{\prime},w^{\prime})$ of Tannakian categories and the pair $(\ca{T},w)$ and $(\Vect_{K^{\prime}},\id)$ of weakly Tannakian categories. Thus both $w \boxtimes w^{\prime}$ and $w \boxtimes \id$ are symmetric strong monoidal functors whose extension to ind-objects are faithful and exact. It follows that
 \[
\Ind(\id \boxtimes w^{\prime}) \colon \Ind(\ca{T} \boxtimes \ca{T}^{\prime}) \rightarrow \Ind(\ca{T} \boxtimes \Vect^{\fd}_{K^{\prime}})
 \]
 is also faithful and exact. Since this functor is strong monoidal, it detects objects with duals. The assumption that $\ca{T} \boxtimes \Vect^{\fd}_{K^{\prime}}$ is rigid therefore implies that $\ca{T} \boxtimes \ca{T}^{\prime}$ is rigid as well. By Lemma~\ref{lemma:rigid_abelian}, it only remains to check that the endomorphism ring of the unit object is a field. Since any module over a field is flat, it follows that the canonical functor
 \[
 \ca{T} \otimes \ca{T}^{\prime} \rightarrow \ca{T} \boxtimes \ca{T}^{\prime}
 \]
 of $k$-linear categories is full and faithful. Thus the endomorphism ring of the unit object of $\ca{T} \boxtimes \ca{T}^{\prime}$ is canonically isomorphic to $k$. This shows that $\ca{T} \boxtimes \ca{T}^{\prime}$ is indeed Tannakian.
 
 The second claim follows from the first and the above discussion of Corollary~\ref{cor:coherent_sheaves_on_products}.
 \end{proof}
 
 From now on, we will only rely on the above corollary to give the desired proof of uniqueness of fiber functors. Note that we have not used the full strength of Theorems~\ref{thm:tannaka_duality} and \ref{thm:kelly_tensor_weakly_tannakian} in its proof. It should thus be possible to convince oneself of the truth of the above corollary by more elementary means. Note, however, that the use of the more general theory makes it possible to avoid certain technical questions. For example, in the above proof, the existence of the Deligne tensor product of two Tannakian categories follows directly from the much simpler existence of the Kelly tensor product.
 
 The reader enticed by the powers of abstract machinery might be interested in the following less elementary, but perhaps more enlightening, proof of Theorem~\ref{thm:kelly_tensor_weakly_tannakian}. By analyzing the existence question for fiber functors on a finitely cocomplete symmetric monoidal $k$-linear category $\ca{A}$ in some detail, one can show that, for any such $\ca{A}$, there exists a \emph{universal} weakly Tannakian category $T(\ca{A})$ equipped with a right exact symmetric strong monoidal $k$-linear functor to $\ca{A}$. In other words, weakly Tannakian categories form a (bicategorically) coreflective subcategory of finitely cocomplete symmetric monoidal $k$-linear categories. Thus the subcategory of weakly Tannakian categories is closed under all colimits, in particular binary coproducts. That the binary coproduct is given by the Kelly tensor product is proved in much the same way as the well-known fact that the coproduct of two commutative rings is given by their tensor product. This is discussed in detail in \cite[\S 5]{SCHAEPPI_COLIMITS}. We will return to this viewpoint in \S \ref{section:universal_tannakian_categories}.

\subsection{Scalar extension of linear categories}\label{section:scalar_extension_of_linear_cats}

 In the above section, we have shown that uniqueness of fibre functors follows from the fact that the scalar extension of a Tannakian category is again Tannakian (see Corollary~\ref{cor:scalar_to_uniqueness}). We therefore need a convenient construction of the scalar extension $\Vect_K^{\fd} \boxtimes \ca{A}$ of a finitely cocomplete $k$-linear category $\ca{A}$. This extension is very simple to describe if the field extension $K$ is \emph{finite}. Namely, the external tensor product $K \odot - \colon \ca{A} \rightarrow \ca{A}$ gives a monad whose category of modules is equivalent to $\Vect_K^{\fd} \boxtimes \ca{A}$. In order to get a similar description for infinite field extensions, we need to pass to ind-objects. Recall that an object $C$ of a category $\ca{C}$ is called \emph{finitely presentable} if the hom-functor $\ca{C}(C,-)$ commutes with filtered colimits. We denote the full subcategory of finitely presentable objects by $\ca{C}_{\fp}$.
 
\begin{prop}\label{prop:scalar_extension}
 Let $\ca{A}$ be a finitely cocomplete $k$-linear category and let $K$ be an arbitrary field extension of $k$. Let $\Ind(\ca{A})_K$ be the category of modules of the monad
 \[
 K \odot - \colon \Ind(\ca{A}) \rightarrow \Ind(\ca{A})
 \]
 on the category of ind-objects of $\ca{A}$. Then there is an equivalence
 \[
  \Vect_K^{\fd} \boxtimes \ca{A} \simeq \bigl( \Ind(\ca{A})_K \bigr)_{\fp}
 \]
 of $K$-linear categories which is natural in $\ca{A}$.
\end{prop}
 
 \begin{proof}
 The Kelly tensor product can be defined relative to any class of colimit shapes. We write $-\boxtimes_c -$ for the Kelly tensor product relative to all small colimits. From a direct inspection of the respective universal properties it follows that passage to ind-objects is compatible with the two tensor products. Explicitly, given two finitely cocomplete $k$-linear categories $\ca{A}$ and $\ca{B}$, there is an equivalence
 \[
 \Ind(\ca{A} \boxtimes \ca{B}) \simeq \Ind(\ca{A}) \boxtimes_c \Ind(\ca{B})
 \]
 of $k$-linear categories which is natural in both $\ca{A}$ and $\ca{B}$.
 
 The equivalence $\ca{A} \simeq \Ind(\ca{A})_{\fp}$ therefore reduces the problem to checking that $\Vect_K \boxtimes_c \ca{C} \simeq \ca{C}_K$ for any cocomplete $k$-linear category $\ca{C}$. This follows as in the case of finitely cocomplete categories and finite field extensions, either by directly checking the universal properties, or from the following argument.
 
 The category $\Vect_K$ is the category of modules of the cocontinuous monad $K \otimes - \colon \Vect_k \rightarrow \Vect_k$. In the 2-category of cocomplete $k$-linear categories and cocontinuous functors, Eilenberg-Moore objects coincide with Kleisli objects, hence they are preserved by any left biadjoint such as $- \boxtimes_c \ca{C}$. Thus the canonical functor
 \[
 \ca{C}_{K} \simeq (\Vect_k \boxtimes_c \ca{C})_{(K \otimes -)\boxtimes \ca{C}} \rightarrow (\Vect_k)_K \boxtimes_c \ca{C} 
 \]
 gives the desired natural equivalence.
 \end{proof}

 Recall that giving an object of a finitely cocomplete $k$-linear category $\ca{A}$ amounts to the same as giving a right exact functor $\Vect_k \rightarrow \ca{A}$. More precisely, evaluation at the one-dimensional vector space gives an equivalence
 \[
 \Rex(\Vect_k, \ca{A}) \rightarrow \ca{A}
 \]
 of $k$-linear categories. Thus it makes sense to speak of the \emph{scalar extension of an object}. Namely, the scalar extension of $A \in \ca{A}$ is the object corresponding to $K$-linear functor
 \[
 \Vect_K^{\fd} \boxtimes F \colon \Vect^{\fd}_K \simeq \Vect_K^{\fd} \boxtimes \Vect_k^{\fd} \rightarrow \Vect_K^{\fd} \boxtimes \ca{A} \smash{\rlap{,}}
 \]
 where $F$ denotes the essentially unique right exact functor $\Vect_k \rightarrow \ca{A}$ for which there exists an isomorphism $F(k) \cong A$. 
 
 The following result shows that scalar extensions of certain simple objects are again simple. Note that this is valid independent of the existence of any tensor structure on the category.
 
 \begin{prop}\label{prop:scalar_extension_of_simple_object}
 Let $S \in \ca{A}$ be a simple object such that the canonical morphism $k \rightarrow \ca{A}(S,S)$ is an isomorphism.  Let $K$ be a field extension of $k$. Then the base change of $S$ is a simple object with endomorphism ring isomorphic to $K$.
 \end{prop}
 
 \begin{proof}
 Under the equivalence of Proposition~\ref{prop:scalar_extension}, the scalar extension of an object $A \in \ca{A}$ is the free $K \odot (-)$-module on $A$ in $\Ind(\ca{A})$. The underlying object of $K \odot S$ in $\Ind(\ca{A})$ is isomorphic to $\bigoplus_{i \in I} S$, where $S$ denotes any set with the cardinality of a $k$-basis of $K$.
 
 In any Grothendieck abelian category, subobjects of semisimple objects are again semisimple, and clearly the simple objects occurring in the subobject are contained in those occuring in the direct sum. Thus the assumption that $S$ is simple implies that any subobject of $K \odot S$ is of the form $\bigoplus_{j \in J} S$ for some subset $J \subseteq I$.
 
 From the isomorphism $k \rightarrow \ca{A}(S,S)$ it follows that the functor
 \[
 -\odot S \colon  \Vect_k \rightarrow \Ind(\ca{A})
 \]
 induces an equivalence between $\Vect_k$ and the full subcategory of $\Ind(\ca{A})$ consisting of direct summands of copies of $S$. Under this equivalence, a subobject $I \subseteq K \odot S$ in $\Ind(\ca{A})_K$ corresponds to an ideal of $K$, so it must either be zero or all of $K$. This shows that the base change of $S$ is simple, and the sequence of isomorphisms
 \[
 \Ind(\ca{A})_K(K \odot S,K \odot S) \simeq \Ind(\ca{A})(S, K \odot S) \simeq K
 \]
 shows that the endomorphism ring of $K \odot S$ is isomorphic to $K$. 
 \end{proof}
 
\section{Scalar extension of Tannakian categories}\label{section:scalar_extension}

\subsection{Some generalities on exterior powers}\label{section:exterior_powers}

 There are two ways to generalize the notion of an exterior power of a vector space to other contexts. Given a module $M$ over a commutative ring $R$, the $n$-th exterior power $E^n(M)$ is traditionally defined as the quotient $M \otimes \ldots \otimes M \slash J$ of the $n$-fold tensor power by the submodule $J$ generated by the elements $m_1 \otimes \ldots \otimes m_n$ with the property that there exists indices $i \neq j$ with $m_i = m_j$. Alternatively, one can define a kind of exterior power $\Lambda^n(M)$ as the image of the \emph{antisymmetrizer}
 \[
 \sum_{\sigma \in \Sigma_n} \sgn(\sigma) \sigma \colon M^{\otimes n} \rightarrow M^{\otimes n} 
 \]
 which sends an elementary tensor $m_1 \otimes \ldots m_n$ to $\sum_{\sigma \in \Sigma_n} \sgn(\sigma) m_{\sigma(1)} \otimes \ldots m_{\sigma(n)}$. There is always a surjective comparison morphism $E^n(M) \rightarrow \Lambda^n(M)$, but there are examples of modules where the comparison is not an isomorphism, see \cite[\S 5]{FLANDERS} for examples. 
 
 The advantage of $\Lambda^n(M)$ is that it can be defined in much greater generality, for example, in any abelian category with a symmetric monoidal structure. On the other hand, the exterior power $E^{n}(M)$ has the following useful property that gives a criterion for a finitely presentable module to be projective.
 
\begin{prop}\label{prop:projective_criterion}
 Let $M$ be a finitely presentable $R$-module. If $M$ is generated by $n$ elements and the exterior power $E^n(M)$ is free of rank one, then $M$ is projective of constant rank $n$.
\end{prop}

\begin{proof}
 This is a special case of \cite[Theorem~2.4]{RAMRAS}. Since we already assume that $M$ is finitely presentable, we can in fact immediately reduce this to the case where $R$ is a local ring, where it follows by induction from the observation that the existence of an epimorphism $M \otimes \ldots \otimes M \rightarrow R$ of $R$-modules implies that $R$ is a direct summand of $M$, see \cite[Theorem~1.4]{RAMRAS}.
\end{proof}

 It is therefore useful to know when the two definitions of exterior powers coincide. Flanders calls a module \emph{regularly $n$-alternate} if the comparison morphism $E^{n}(M) \rightarrow \Lambda^{n}(M)$ is an isomorphism, respectively \emph{regularly alternate} if it is regularly $n$-alternate for all $n$, see \cite[\S 3]{FLANDERS}. It is elementary to check that a finitely generated free module is regularly alternate. In \S \ref{section:universal_tannakian_categories}, we will need the following result.
 
\begin{prop}\label{prop:regularly_alternate}
 Let $M$ be an $R$-module. If $M$ is generated by $n$ elements, then $M$ is regularly $n$-alternate. 
\end{prop}

\begin{proof}
 This is for example proved in \cite[Theorem~5]{FLANDERS}.
\end{proof}
 
 A much more intricate analysis of the antisymmetrizer shows that a direct sum of regularly alternate modules is regularly alternate, see \cite[Theorem~6]{FLANDERS}. Together with the above proposition, this implies that any finite direct sum of cyclic modules is regularly alternate. This applies in particular to all the modules that appear in \S \ref{section:proof_of_scalar_extension} below, so in principle, we do not need to distinguish $E^{r}(M)$ and $\Lambda^n(M)$ in our elementary proof. However, the proof does not rely on knowing this fact.
 
\subsection{An elementary proof of the scalar extension result}\label{section:proof_of_scalar_extension}

 The aim of this section is to give an elementary proof of the following theorem.
 
 \begin{thm}\label{thm:scalar_extension}
  Let $k$ be a field, $\ca{T}$ a Tannakian category over $k$, and let $k^{\prime}$ be any field extension of $k$. Then the scalar extension $\Vect_{k^{\prime}}^{\fd} \boxtimes \ca{T}$ is a Tannakian category over $k^{\prime}$.
 \end{thm}
 
 As we have already observed, an immediate corollary of this theorem is that fiber functors of Tannakian categories are locally unique.
 
 \begin{cor}
 Let $\ca{T}$ be a Tannakian category over $k$ and $S$ a scheme. Then for any two fiber functors $w_1, w_2 \colon \ca{T} \rightarrow \QCoh(S)$, there exists an $\fpqc$-cover $p \colon T \rightarrow S$ such that there exists a natural isomorphism $p^{\ast} w_1 \cong p^{\ast} w_2$ of symmetric monoidal functors. Moreover, there exists an $\fpqc$-gerbe $G$ over $k$ with affine diagonal and which admits an $\fpqc$-cover by an affine scheme such that $\ca{T}$ is equivalent to $\Coh(G)$ as a symmetric monoidal $k$-linear category.
 \end{cor}
 
 \begin{proof}
 By Theorem~\ref{thm:scalar_extension}, the condition of Corollary~\ref{cor:scalar_to_uniqueness} is always satisfied. This proves the first claim, and the second follows directly from this and Tannaka duality (see Theorem~\ref{thm:tannaka_duality}).
 \end{proof}
 
 To prove Theorem~\ref{thm:scalar_extension}, first note that we have already seen that the scalar extension of a Tannakian category has a simple unit object with endomorphism ring the (new) ground field (see Proposition~\ref{prop:scalar_extension_of_simple_object}). Moreover, from Theorem~\ref{thm:kelly_tensor_weakly_tannakian}, we know that the scalar extension is a \emph{weakly} Tannakian category and that it has a fiber functor with target the category of $k^{\prime} \ten{k} K$-modules. In other words, we want to show that any weakly Tannakian category $\ca{A}$ with fiber functor
 \[
 w \colon \ca{A} \rightarrow \Mod_{k^{\prime} \ten{k} K}
 \]
 and a simple unit object is Tannakian. To do this, we only need to check that every object of the scalar extension has a dual, see Lemma~\ref{lemma:rigid_abelian}.
 
 It therefore suffices to prove Theorem~\ref{thm:scalar_extension} for finitely generated field extensions since scalar extension commutes with filtered (bicategorical) colimits. This naturally reduces the problem to three cases, namely the case where the field extension $k^{\prime}$ of $k$ is finite separable, finite purely inseparable, or purely transcendental of transcendence degree one.
 
 The case of finite separable field extensions is the easiest. In this case, the algebra $k^{\prime} \ten{k} K$ is a semisimple algebra, so all modules are projective. The claim follows since fiber functors detect objects with duals.
 
 In the finite purely inseparable case, we can further reduce to field extensions $k^{\prime}$ generated by a single $p$-th root, where $p$ denotes the characteristic of $k$. Then the tensor product $k^{\prime} \ten{k} K$ is either again a field, in which case there is nothing to do, or it is isomorphic to the quotient $K[t] \slash (t^p)$ of the polynomial algebra in one variable. This case was proved by Deligne, see \cite[Th{\'e}or{\`e}me~5.4]{DELIGNE_SEMI_SIMPLE}. Since this is in particular a local ring, it also follows from Lemma~\ref{lemma:local_ring} below.
 
 It remains to consider the case where $k^{\prime}=k(t)$, the field of rational functions in one variable. As a $k$-algebra, $k(t)$ is given by the directed union of the algebras $k[t]_{f}$ of localizations of the ring of polynomial functions at the non-zero polynomials $f$, ordered by divisibility. Thus the tensor product $k(t) \ten{k} K$ is isomorphic to the directed union of $K$-algebras $K[t]_f$, where $f$ again denotes a non-zero polynomial with coefficients in $k$. These are all principal ideal domains, so it follows that any finitely \emph{presentable} module over  $k(t) \ten{k} K$ is a direct sum of cyclic modules. In particular, any such module is a direct sum of a finitely generated free module and a torsion module, so the following lemma is applicable.
 
\begin{lemma}\label{lemma:exterior_power_torsion}
 Let $R$ be a domain, let $F_n$ be a free $R$-module of rank $n$, and let $T$ be a torsion $R$-module. Then the module $\Lambda^{n+1}(F_{n} \oplus T)$ is a non-zero torsion module.
\end{lemma}

\begin{proof}
 To see that $\Lambda^{n+1}(F_n \oplus T)$ is torsion, it suffices to check that the exterior power $E^{n+1}(F_n \oplus T)$ is torsion since the canonical morphism $E^n(M) \rightarrow \Lambda^n(M)$ is an epimorphism for any $R$-module $M$. For $E^{n+1}(-)$, the direct sum formula
 \[
 E^{n+1}(F_n \oplus T) \cong \bigoplus_{k + \ell = n+1} E^k(F_n) \otimes E^{\ell}(T)
 \]
 holds. Since $E^{n+1}(F_n) \cong 0$, all the direct summands occuring are direct sums of copies of $E^{\ell}(T)$ with $\ell \geq 1$, hence torsion.
 
 It remains to check that $\Lambda^n(F_{n+1} \oplus T)$ is non-zero. To see this, pick $x \in T$ a non-zero element and let $e_i$ denote a basis of $F_n$. Let $v_i=e_i$, $1 \leq i \leq n$, and let $v_{n+1}=x$. We claim that the image
\begin{equation}\label{eqn:antisymmetrizer}
a(v_1 \otimes \ldots \otimes v_{n+1}) \defl \sum\nolimits_{\sigma \in \Sigma_{n+1}} \sgn(\sigma) v_{\sigma(1)} \otimes \ldots \otimes v_{\sigma(n+1)}
\end{equation}
  of $v_1 \otimes \ldots \otimes v_{n+1}$ under the antisymmetrizer on $(F_n \oplus T)^{\otimes n+1}$ is non-zero. To see this, consider the homomorphism
 \[
 \varphi \defl \pr_1 \otimes \ldots \otimes \pr_n \otimes \pr_T \colon (F_n \oplus T)^{\otimes n+1} \rightarrow R \otimes \ldots R \otimes T \cong T
 \]
 where $\pr_i$ denotes the projection to the $i$-coordinate of $F_n \cong R^{\oplus n}$. The image of any summand occurring in Formula~\eqref{eqn:antisymmetrizer} is thus of the form
 \[
 \sgn(\sigma) \pr_1(v_{\sigma(1)}) \otimes \ldots \otimes \pr_n(v_{\sigma(n)}) \otimes \pr_T(v_{\sigma(n+1)})
 \]
 for some permutation $\sigma$ of $n+1$ elements. This vanishes as soon as $v_{\sigma(i)} \neq v_i$ for some $1 \leq i \leq n$, hence for any $\sigma \neq \id$. Thus we have
 \[
 \varphi\bigl( a(v_1 \otimes \ldots \otimes v_{n+1}) \bigr)=1 \otimes \ldots \otimes 1 \otimes x \smash{\rlap{,}}
 \]
 which is non-zero since $x \neq 0$.
\end{proof}

\begin{lemma}\label{lemma:fiber_functor_to_filtered_colimit_of_PIDs}
 Let $\ca{A}$ be a weakly Tannakian category with a simple unit object. If there exists a fiber functor
 \[
  w \colon \ca{A} \rightarrow \Mod_R
 \]
 where $R$ is a directed union of principal ideal domains, then $\ca{A}$ is Tannakian.
\end{lemma}

\begin{proof}
 The endomorphism ring of the unit of a monoidal category is always commutative by the Eckmann--Hilton argument. The assumption that the unit is simple thus implies that the endomorphism ring is a field. To show that $\ca{A}$ is Tannakian, it only remains to show that it is rigid (see Lemma~\ref{lemma:rigid_abelian}). To see this, it suffices to check that $w(A)$ is finitely generated projective for any object $A \in \ca{A}$.
 
 Since the objects with duals form a generator of $\Ind(\ca{A})$ (see Definition~\ref{dfn:weakly_tannakian}), it follows that $w(A)$ is a finitely presentable module. The assumption on $R$ therefore implies that it is a direct sum of finitely many cyclic modules. In particular, we have $w(A) \cong F_n \oplus T$ where $F_n$ is a free module of rank $n$ and $T$ is a finitely presentable torsion module. It only remains to check that $T \cong 0$.
 
 If not, we claim that there exists an object $X \in \ca{A}$ such that $w(X)$ is a non-zero torsion module. Indeed, $w\bigr(\Lambda^{n+1}(A)\bigr) \cong \Lambda^{n+1}\bigr(w(A)\bigl)\cong \Lambda^{n+1}(F_n \oplus T)$ is a non-zero torsion module by Lemma~\ref{lemma:exterior_power_torsion}. It thus only remains to check that the implication
 \[
 w(X) \; \text{torsion} \Rightarrow w(X) \cong 0
 \]
 holds for all $X \in \ca{A}$.
 
 We thus fix an object $X$ such that $T \defl w(X)$ is torsion. By definition of weakly Tannakian categories, there exists an epimorphism $p \colon V \rightarrow X$ where $V \in \ca{A}$ is an object with a dual. Let $\varphi \colon V^{\prime} \rightarrow V$ denote the kernel of $p$ in $\Ind(\ca{A})$. By assumption on $R$, the morphism $w(p) \colon w(V) \rightarrow w(X)$ is extended from one of the principal ideal domains $R_0 \subseteq R$ exhibiting $R$ as directed union of such. Since $R$ is flat over $R_0$ (it is clearly torsion free as $R_0$-module), it follows that the kernel of $w(p)$ is also extended from $R_0$, hence it is finitely generated and free. On the other hand, this kernel is isomorphic to the image of $V^{\prime}$ under the extension of $w$ to ind-objects. It follows that $V^{\prime}$ has a dual, hence in particular that it lies in $\ca{A}$. Moreover, passing to the field of fractions shows that $w(V^{\prime})$ and $w(V)$ are free of the same rank $n \geq 0$.
 
 We claim that $\varphi$ is an isomorphism. It suffices to check that $w(\varphi)$ is an isomorphism. Since this is a morphism between free modules of the same rank, we only need to check that the determinant $E^n\bigl(w(\varphi)\bigr) \cong \Lambda^n\bigl(w(\varphi)\bigr) \cong w\bigr(\Lambda^n(\varphi)\bigl)$ is invertible (recall that the two notions of exterior power coincide for free modules). Since $w(V^{\prime})$ and $w(V)$ are flat, it follows that the $n$-fold tensor product $\varphi^{\otimes n} \colon (V^{\prime})^{\otimes n} \rightarrow V^{\otimes n}$ is a monomorphism. Thus $\Lambda^n(\varphi)$ is a monomorphism between two ind-objects whose images in $\Mod_R$ are free of rank one. But any such ind-object is $\otimes$-invertible in $\Ind(\ca{A})$ (hence it lies in particular in $\ca{A}$). Moreover, since the unit object of $\ca{A}$ is simple, all $\otimes$-invertible objects of $\ca{A}$ are simple as well. Thus $\Lambda^n(\varphi)$ is indeed an isomorphism. This shows that $T=w(X)$ is zero, as claimed.
 \end{proof}
 
 With this in hand, the proof that a scalar extension of Tannakian categories is again Tannakian is now straightforward.
 
 \begin{proof}[Proof of Theorem~\ref{thm:scalar_extension}]
  Let $\ca{T}$ be a Tannakian category over $k$ with fiber functor $w \colon \ca{A} \rightarrow \Vect_K$. For any field extension $k^{\prime}$ of $k$, the scalar extension $\Vect_{k^{\prime}}^{\fd} \boxtimes \ca{T}$ is weakly Tannakian and
  \[
  \id \boxtimes w \colon \Vect_{k^{\prime}}^{\fd} \boxtimes \ca{T} \rightarrow \Mod_{k^{\prime} \ten{k} K}
  \]
  is a fiber functor. Moreover, the unit object of $\Vect_{k^{\prime}}^{\fd} \boxtimes \ca{T}$ is simple by Proposition~\ref{prop:scalar_extension_of_simple_object}. By Lemma~\ref{lemma:rigid_abelian}, it only remains to show that the scalar extension is rigid.

 Since the Kelly tensor product commutes with (filtered) bicategorical colimits, it suffices to check this for finitely generated field extensions. By adjoining one element at a time, this further reduces the problem to the case of a finite separable extension, a purely inseparable extension of degree equal to the characteristic, and a purely transcendental extension of transcendence degree one.
  
  The first case is straightforward and the second is covered either by \cite[Th{\'e}or{\`e}me~5.4]{DELIGNE_SEMI_SIMPLE} or Lemma~\ref{lemma:local_ring} below. If $k^{\prime}$ is purely transcendental of transcendence degree one, then $k^{\prime} \ten{k} K$ is the directed union of the localizations $K[t]_f$ of the ring of polynomials at the non-zero polynomials with coefficients in $k$. Thus $\Vect_{k^{\prime}}^{\fd} \boxtimes \ca{T}$ is Tannakian by Lemma~\ref{lemma:fiber_functor_to_filtered_colimit_of_PIDs}.
 \end{proof}
\section{Universal Tannakian categories}\label{section:universal_tannakian_categories}

\subsection{Weakly Tannakian categories with simple unit object}\label{section:simple_unit}
 In Lemma~\ref{lemma:fiber_functor_to_filtered_colimit_of_PIDs}, we have seen that a weakly Tannakian category with a simple unit object and which admits a particular type of fiber functor is necessarily Tannakian. The aim of this section is to prove the following theorem, which shows that the condition on the fiber functor is in fact superfluous.
 
\begin{thm}\label{thm:simple_unit_iff_tannakian}
 A weakly Tannakian category $\ca{A}$ is Tannakian if and only if the unit object of $\ca{A}$ is simple.
\end{thm}

 The proof of this theorem is split into two parts. We first prove this in the special case where $\ca{A}$ admits a fiber functor whose target is a local ring. In the second part, we show that any weakly Tannakian category with a simple unit object admits such a fiber functor. The first part generalizes the argument of \cite[Th{\'e}or{\`e}me~5.4]{DELIGNE_SEMI_SIMPLE}, while the second is more technical and relies on the notion of Adams algebra introduced in \cite[\S 4.1]{SCHAEPPI_COLIMITS}.
 
 \begin{lemma}\label{lemma:local_ring}
 Let $\ca{A}$ be a weakly Tannakian category such that there exists a fiber functor $w \colon \ca{A} \rightarrow \Mod_R$ where $(R,\mathfrak{m})$ is a local ring. Then $\ca{A}$ is a Tannakian category.
 \end{lemma}

\begin{proof}
 The endomorphism ring of the unit of a monoidal category is always commutative by the Eckmann--Hilton argument. The assumption that the unit is simple thus implies that the endomorphism ring is a field. To show that $\ca{A}$ is Tannakian, it only remains to show that it is rigid (see Lemma~\ref{lemma:rigid_abelian}). This amounts to checking that the $R$-module $M \defl w(A)$ is finitely generated free for any $A \in \ca{A}$. Since $M$ is finitely presentable, the quotient $M \slash \mathfrak{m} M$ is a $k$-dimensional vector space over the residue field of $R$ for some $k \geq 1$. If $M \neq 0$, which we henceforth assume, then we have $k > 0$.
 
 By Nakayama, $M$ is generated by $k$ elements and thus the $k$-th exterior power $E^{k}(M)$ is a cyclic module. Note that $R \slash \mathfrak{m} \ten{R} E^{k}(M) \cong R \slash \mathfrak{m}$, so $E^{k}(M)$ is in particular non-zero.
 
 Recall from Proposition~\ref{prop:regularly_alternate} that any module generated by $k$ elements is regularly $k$-alternate, that is, the comparison morphism $E^{k}(M) \rightarrow \Lambda^{k}(M)$ is an isomorphism. Moreover, from Proposition~\ref{prop:projective_criterion} we know that $M$ is projective (hence free) if $E^{k}(M)$ is. We have thus reduced the problem to checking that $\Lambda^{k}(M) \cong \Ind(w)\bigl(\Lambda^{k}(A) \bigr)$ is free of rank one, where $\Ind(w)$ denotes the extension of $w$ to ind-objects. Note that, a priori, we only know that $\Lambda^{k}(M)$ is finitely generated (as opposed to finitely presentable). Nevertheless, since $\Lambda^{k}(M) \cong R \slash I$ for some (possibly infinitely generated) ideal $I \subseteq R$, it only remains to check that the annihilator of $\Lambda^{k}(M)$ is trivial.
 
 To see this, recall that $\Ind(\ca{A})$ is a symmetric monoidal closed category since tensoring with any object commutes with small colimits. We denote the unit object by $O$ and the internal hom by $[-,-]$. We let $W \defl \Ind(w) \colon \Ind(\ca{A}) \rightarrow \Mod_R$ denote the extension of $w$ to ind-objects. Using this notation, our aim is to prove that $W\bigl(\Lambda^{k}(A) \bigr)$ has trivial annihilator. By construction, we have an epimorphism $A^{\otimes k} \rightarrow \Lambda^{k}(A)$, so it suffices to check that $W(X)$ has a trivial annihilator for any non-zero finitely generated object $X \in \Ind(\ca{A})$.
 
 Any symmetric strong monoidal functor between symmetric monoidal closed categories such as $W$ is in particular \emph{lax closed}. Thus there is a natural transformation $\overline{W} \colon W[X,Y] \rightarrow \Hom_R(WX,WY)$ subject to various coherence axioms, one of which says that the diagram
 \[
 \xymatrix{
 W(O) \ar[d]_{\cong} \ar[r]^-{W(j)} & W[X,X] \ar[d]^{\overline{W}} \\
 R \ar[r] & \Hom_R(WX,WX) 
 }
 \]
 is commutative for any $X \in \Ind(\ca{A})$, where $j$ denotes the adjunct of the unit isomorphism $X \otimes O \cong X$. Since the unit object $O$ is simple, the kernel of $j$ must be either zero or $O$. The latter happens if and only if $X \cong 0$. Thus, for non-zero $X$, the morphism $W(j)$ is a monomorphism. This reduces the problem to checking that $\overline{W} \colon W[X,Y] \rightarrow \Hom(WX,WY)$ is a monomorphism for any finitely generated $X \in \Ind(\ca{A})$.
 
 By Definition~\ref{dfn:weakly_tannakian}, there exists an epimorphism $p \colon X^{\prime} \rightarrow X$ for some object $X^{\prime} \in \ca{A}$ with a dual. Since $W$ is exact, both horizontal morphisms in the naturality diagram
 \[
 \xymatrix{
 W[X,Y] \ar[r]^-{W(p^{\ast})} \ar[d]_{\overline{W}} & W[X^{\prime},Y] \ar[d]^{\overline{W}} \\
 \Hom_R(WX,WY) \ar[r]_{(Wp)^{\ast}} & \Hom_R(WX^{\prime},WY)
 }
 \]
 are monomorphisms. This reduces the problem to the case where $X$ itself has a dual. But in this case, it is well known that the comparison homomorphism $\overline{W} \colon W[X,Y] \rightarrow \Hom_R(WX,WY)$ is in fact an isomorphism.
\end{proof}

\begin{rmk}
 The only place in the above proof where we made use of the fact that $R$ is local was in the implication that an $R$-module $M$ with minimal generating set of finite cardinality $k>0$ satisfies $E^k(M) \neq 0$. This does not hold for more general rings, even if $M$ is locally free of finite rank (for example, if $M$ is a non-principal ideal in a Dedekind domain). 
\end{rmk}

 The following proposition is closely related to \cite[Theorem~1.3.1]{SCHAEPPI_STACKS}. We will apply it in the case where $\ca{C}=\Ind(\ca{A})$ for a weakly Tannakian category $\ca{A}$.

 \begin{prop}\label{prop:filtered_colimit_of_duals}
 Let $\ca{C}$ be an abelian locally finitely presentable symmetric monoidal closed category whose unit object is finitely presentable. Assume that the full subcategory $\ca{C}^{d}$ of objects with a dual is a generator of $\ca{C}$. Let $B$ be a commutative ring and let $F \colon \ca{C} \rightarrow \Mod_B$ be a symmetric strong monoidal left adjoint, with right adjoint $U$. If $F$ is exact, then $U(B)$ is a filtered colimit of objects with duals.
 \end{prop}

\begin{proof}
 Write $w \colon \ca{C}^d \rightarrow \Mod_B$ for the restriction of $F$ to $\ca{C}^d$. Since the objects with duals form a dense generator by \cite[Theorem~2 and Example~(3)]{DAY_STREET_GENERATORS}, we only need to check that the slice category $\bigl(\ca{C}^d \downarrow U(B) \bigr)$ is filtered. Since $F$ is left adjoint to $U$, it follows that this slice category is equivalent to $(w \downarrow B)$. Moreover, since $w$ takes values in finitely generated projective $B$-modules, we obtain an equivalence $(w \downarrow B)\simeq \bigl(B \downarrow w(-)^{\vee}\bigr)^{\op} \simeq \el\bigl(w(-)^{\vee}\bigr)$ by composition with the contravariant equivalence $(-)^{\vee}$ which takes a finitely generated projective $B$-module to its dual. It thus remains to check that the category of elements of $w(-)^{\vee}$ is cofiltered.
 
 It is non-empty since it contains the zero-element of the zero object, and it clearly has finite direct sums. It only remains to check that any pair of morphisms in the category of elements is equalized on the left by some (not necessarily unique) morphism. This boils down to checking that any element in the kernel of $w(f)^{\vee} \colon w(V^{\prime})^{\vee} \rightarrow w(V)^{\vee}$ (where $f$ is an arbitrary morphism $V \rightarrow V^{\prime}$) lies in the image of the morphism $w(g)^{\vee} \colon w(V^{\prime\prime})^{\vee} \rightarrow w(V^{\prime})^{\vee}$ for some $g \colon V^{\prime} \rightarrow V^{\prime \prime}$. This follows from the exactness of $F$ and the fact that the duals form a generator of $\ca{C}$, which implies that they jointly surject onto the kernel of $f^{\vee}$.
\end{proof}

 Recall that we call a commutative algebra $A$ in a Grothendieck abelian symmetric monoidal closed category $(\ca{C},\otimes,O)$ an \emph{Adams algebra} if the unit $\eta \colon O \rightarrow A$ can be written as a filtered colimit of morphisms $\eta_i \colon O \rightarrow A_i$, $i \in I$, where each $A_i$ has a dual and the dual morphism $\eta_i^{\vee} \colon A_i^{\vee} \rightarrow O$ are epimorphisms for all $i \in I$. Using exactness of filtered colimits and the internal hom, one can show that any Adams algebra is faithfully flat, that is, the functor $A \otimes - \colon \ca{C} \rightarrow \ca{C}$ is faithful and exact (see \cite[\S 4.1]{SCHAEPPI_COLIMITS} for details).
 
 \begin{cor}\label{cor:adams}
 Let $\ca{C}$ be as in Proposition~\ref{prop:filtered_colimit_of_duals}, and let $B$ be a non-zero commutative ring. If the unit object $O$ of $\ca{C}$ is simple, then any exact symmetric strong monoidal left adjoint $F \colon \ca{C} \rightarrow \Mod_B$ is faithful.
 \end{cor}
 
 \begin{proof}
 Let $U$ be the right adjoint of $F$ and let $A \defl U(B)$ be the commutative algebra in $\ca{C}$ obtained from the trivial $B$-algebra $B$. By \cite[Proposition~3.9]{SCHAEPPI_GEOMETRIC}, the functor $F$ is equivalent to the free $A$-module functor, so it suffices to check that $A \otimes -$ is faithful and exact.
 
 From Proposition~\ref{prop:filtered_colimit_of_duals}, it follows that the unit of $A$ is a filtered colimit of morphisms $\eta_i \colon O \rightarrow A_i$, $i \in I$, where each $A_i$ has a dual. We claim that the dual morphisms $\eta_i^{\vee}$ are epimorphisms.
 
 Since $O$ is simple, $\eta_i^{\vee}$ is either an epimorphism or zero. Assume $\eta_i^{\vee}=0$. Then its dual $\eta_i$ would also be equal to zero. Since the unit of $A$ factors through $\eta_i$, this implies that the unit of $A$ is zero and thus that $A \cong 0$. This contradicts the fact that the category of $A$-modules in $\ca{C}$ is equivalent to the non-zero category $\Mod_B$. Thus $\eta_i^{\vee}$ is indeed an epimorphism for all $i \in I$. Therefore $A$ is an Adams algebra, so $A \otimes - \colon \ca{C} \rightarrow \ca{C}$ is faithful by \cite[Proposition~4.1.6]{SCHAEPPI_COLIMITS}.
 \end{proof}
 
 With this in hand, it is now easy to show that a weakly Tannakian category with simple unit object is Tannakian.
 
 \begin{proof}[Proof of Theorem~\ref{thm:simple_unit_iff_tannakian}]
 Let $\ca{A}$ be a weakly Tannakian category with simple unit object and fiber functor $w \colon \ca{A} \rightarrow \Mod_B$. Since the unit object is in particular non-zero, we must have $B \neq 0$ by faithfulness of $w$. Let $R=B_{\mathfrak{p}}$ be the localization of $B$ at any prime ideal $\mathfrak{p} \subsetneq B$. Then the composite
 \[
 \xymatrix{\Ind(\ca{A}) \ar[r]^-{\Ind(w)} & \Mod_B \ar[r]^-{R \ten{B} -} & \Mod_R }
 \]
 is an exact symmetric strong monoidal left adjoint. By Corollary~\ref{cor:adams}, it is therefore also faithful. Thus $\ca{A}$ admits a fiber functor with target the category of modules over a local ring. From Lemma~\ref{lemma:local_ring} we know that any weakly Tannakian category with a simple unit object which admits such a fiber functor is Tannakian.
 \end{proof}

\subsection{Applications}\label{section:applications}

 We conclude with two applications of the new characterization of Tannakian categories given in Theorem~\ref{thm:simple_unit_iff_tannakian}. We call an $\fpqc$-gerbe a \emph{Tannakian gerbe} if it has an affine diagonal and an $\fpqc$-cover by an affine scheme. These are precisely the gerbes $G$ for which $\Coh(G)$ is a Tannakian category.
 
 \begin{prop}
 Any minimal closed substack $Z \subseteq X$ of an Adams stack $X$ is a Tannakian gerbe. In particular, any non-empty Adams stack contains a Tannakian gerbe as a closed substack.
 \end{prop}
 
 \begin{proof}
 It suffices to show that the category $\QCoh_{\fp}(Z)$ is Tannakian. By Theorem~\ref{thm:simple_unit_iff_tannakian}, this amounts to checking that the unit object $O_Z$ is simple. Equivalently, we need to show that the ideal defining $Z$ is a maximal subobject of $O_X$. If it were not, then $Z$ would contain a proper non-empty closed substack, contradicting the minimality of $Z$.
 
 The second claim follows from the existence of maximal subobjects of $O_X$, which is a consequence of Zorn's lemma and the fact that $O_X$ is a finitely presentable object of $\QCoh(X)$.
 \end{proof}
 
 The second application concerns the universal weakly Tannakian categories introduced in \cite[\S 5.6]{SCHAEPPI_COLIMITS}. From now on, we fix a commutative ring $R$ and assume that all categories and functors are $R$-linear unless stated otherwise. A finitely cocomplete symmetric monoidal category $\ca{B}$ is called \emph{right exact} if tensoring with any fixed object preserves finite colimits. For any such category $\ca{B}$, there exists a weakly Tannakian category $T(\ca{B})$ and a symmetric strong monoidal linear functor $E \colon T(\ca{B}) \rightarrow \ca{B}$ which is universal in the sense that any symmetric strong monoidal functor $\ca{A} \rightarrow \ca{B}$ factors essentially uniquely through $E$. Our second application gives conditions for the universal weakly Tannakian category to be Tannakian.
 
 We briefly recall the construction of $T(\ca{B})$. To do this, we need the notion of a \emph{locally free object of constant finite rank} and a \emph{locally split epimorphism}. We define these for locally finitely presentable abelian symmetric monoidal categories $\ca{C}$ (and then apply this in the case $\ca{C}=\Ind(\ca{B})$; in particular, we make the additional assumption that $\ca{B}$ is ind-abelian). 
 
 An object $V \in \ca{C}$ with a dual is called \emph{locally free of rank $d$} for some $d \in \mathbb{N}$ if there exists a faithfully flat commutative algebra $B \in \ca{C}$ such that $B \otimes V$ is isomorphic to $B^{\oplus d}$ in the category of $B$-modules in $\ca{C}$. A morphism $p \colon V \rightarrow W$ is called a \emph{locally split epimorphism} if there exists an algebra $B$ as above such that $B \otimes p$ is a split epimorphism of $B$-modules in $\ca{C}$. Locally split right exact sequences are defined similarly, see \cite[Definition~5.4.1]{SCHAEPPI_COLIMITS} for a precise definition.
 
 We write $\ca{B}_{\ell f}$ for the full subcategory of $\ca{B}$ consisting of the locally free objects of $\Ind(\ca{B})$ of constant finite rank and we let $\Sigma$ be the set of locally split right exact sequences in $\Ind(\ca{B})$ between objects in $\ca{B}_{\ell f}$. Then the category $\Lex_{\Sigma}[\ca{B}_{\ell f}^{\op},\Mod_R]$ of presheaves on $\ca{B}_{\ell f}$ which send sequences in $\Sigma$ to  left exact sequences of $R$-modules is reflective in the category of all presheaves. In fact, the locally split epimorphisms induce an enriched Grothendieck topology on the category of $R$-linear presheaves whose category of $R$-linear sheaves is precisely $\Lex_{\Sigma}[\ca{B}_{\ell f}^{\op},\Mod_R]$. The category $T(\ca{B})$ is the closure of $\ca{B}_{\ell f}$ in $\Lex_{\Sigma}$ under finite colimits, see \cite[Definition~5.6.3]{SCHAEPPI_COLIMITS}.
 
 \begin{thm}\label{thm:universal}
 Let $\ca{B}$ be an ind-abelian right exact symmetric monoidal $R$-linear category. If the unit object of $\ca{B}$ is simple, then the universal weakly Tannakian category $T(\ca{B})$ is Tannakian. Moreover, the functor $E \colon T(\ca{B}) \rightarrow \ca{B}$ is exact, full and faithful, and its essential image consists of the locally free objects of constant rank $\ca{B}_{\ell f} \subseteq \ca{B}$.
 \end{thm}
 
 Note that the endomorphism ring of the unit object is necessarily a field if it is simple, so we can assume that $R$ is a field without loss of generality. If this field has characteristic zero, then the above theorem is relatively straightforward to prove, but it is more involved in positive characteristic. We will use McCoy's theorem.
 
 \begin{thm}\label{thm:mccoy}
  Let $B$ be a commutative ring, $f \in B[ x_1, \ldots , x_n ] $ a polynomial in $n$ variables. If $f$ is a zero-divisor in $B[ x_1, \ldots , x_n ]$, then there exists a non-zero element $b \in B$ which annihilates $f$.
 \end{thm}
 
\begin{proof}
 This is proved in \cite[Theorem~3]{MCCOY}.
\end{proof}

 We will also need the following criterion for a morphism with target the unit object to be a locally split epimorphism. It is only applicable in the case where $\ca{C}=\Ind(\ca{B})$ is abelian, which explains why we need to assume that $\ca{B}$ is ind-abelian in Theorem~\ref{thm:universal}.
 
 \begin{lemma}\label{lemma:locally_split}
 Let $\ca{C}$ be a locally finitely presentable abelian symmetric monoidal closed category, $V \in \ca{C}$ a locally free object of rank $d \in \mathbb{N}$. Let $p \colon V \rightarrow O$ be an epimorphism. Then $p$ is locally split if and only if the duals of the morphisms
 \[
 \xymatrix{ O \ar[r]^-{p^{\otimes i}} & V^{\otimes i} \ar[r] & \Sym^{i}(V) }
 \]
 are epimorphisms for all $i \in \mathbb{N}$.
 \end{lemma}
 
 \begin{proof}
  This is precisely the conclusion of \cite[Proposition~5.3.4(iv)]{SCHAEPPI_COLIMITS}, so we only need to check that its conditions are satisfied, namely that $\Sym^{i}(M^{\vee})$ has a dual for all $i \in \mathbb{N}$. Since $M^{\vee}$ is locally free of rank $d$, it is isomorphic to the image of the standard representation of $\mathrm{GL}_d$ under some symmetric strong monoidal functor $\Rep(\mathrm{GL}_d) \rightarrow \ca{C}$ (see \cite[Theorem~4.3.10]{SCHAEPPI_COLIMITS}). This reduces the problem to checking that symmetric powers of the standard representation are locally free, which follows directly from the fact that symmetric powers of free modules are free. 
 \end{proof}
 
 \begin{proof}[Proof of Theorem~\ref{thm:universal}]
  To show that the universal weakly Tannakian category $T(\ca{B})$ is Tannakian, it suffices to check that its unit object is simple (see Theorem~\ref{thm:simple_unit_iff_tannakian}). This category has a generator given by the representable presheaves $\ca{B}(-,V)$ with $V \in \ca{B}_{\ell f}$. We therefore need to show that all morphisms
\[
  \ca{B}(-,p) \colon \ca{B}(-,V) \rightarrow \ca{B}(-,O)
\] 
 are either zero or epimorphisms in $\Lex_{\Sigma}$. Since $\ca{B}(-,p)$ is an epimorphism in $\Lex_{\Sigma}$ if and only if $p$ is a locally split epimorphism in $\Ind(\ca{B})$ (see \cite[Proposition~5.4.13]{SCHAEPPI_COLIMITS}), this amounts to showing that any non-zero morphism $p \colon V \rightarrow O$ is a locally split epimorphism.
  
 Since the unit object is simple and any epimorphism is locally split in characteristic zero (see \cite[Corollary~5.3.5]{SCHAEPPI_COLIMITS}), the proof is complete if $\End(O)$ has characteristic zero. In the general case, we will apply Lemma~\ref{lemma:locally_split}. Since $O$ is simple, this amounts to checking that the the dual of the morphism
\[
  \xymatrix{ O \ar[r]^-{p^{\otimes i}} & V^{\otimes i} \ar[r] & \Sym^{i}(V) }
\]
 is non-zero for all $i \in \mathbb{N}$ and all non-zero morphisms $p \colon V \rightarrow O$. This is clearly the case if and only if the morphism itself is non-zero. Note that the dual $p^{\vee} \colon O \rightarrow V^{\vee}$ of $p$ is a monomorphism since $O$ is simple.
 
 To check that a morphism is non-zero, it suffices to check that its image under some faithful functor is non-zero. By definition of locally free objects of finite rank, we can thus assume that $V \cong O^{\oplus d}$ at the expense of working in some symmetric monoidal category $\ca{D}$ whose unit object $O$ is no longer simple. However, the property that the dual morphism $p^{\vee} \colon O \rightarrow O^{\oplus d}$ is a monomorphism is retained since the base change functor appearing in the definition of locally free objects of rank $d$ commutes with finite limits.
 
 Since the canonical symmetric strong monoidal left adjoint
 \[
 \Mod_{\End(O)} \rightarrow \ca{D}
 \]
 is full and faithful on finitely generated free modules, we can assume without loss of generality that $\ca{D} \simeq \Mod_B$ where $B=\End(O)$. In this case, the claim we want to prove is a statement about polynomials in $d$ variables. Namely, we want to show that, given elements $b_1, \ldots , b_d \in B$ for which the induced morphism $B \rightarrow B^{d}$ is a monomorphism, the $i$-fold power $(\sum b_k x_k)^{i}$ is non-zero in the polynomial ring $B[x_1 , \ldots , x_d]$. 
 
 If this power were zero, the polynomial would in particular be a zero-divisor, so McCoy's Theorem (see Theorem~\ref{thm:mccoy}) would be applicable. But the existence of a non-zero $b \in B$ annihilating $\sum b_k x_k$ contradicts the fact that the morphism $B \rightarrow B^{\oplus d}$ given by the elements $b_k$ is a monomorphism. This concludes the proof that any non-zero morphism $p \colon V \rightarrow O$ in $\Ind(\ca{B})$ where $V$ is locally free of constant finite rank is a locally split epimorphism. Thus the unit of $T(\ca{B})$ is indeed simple, as claimed, and $T(\ca{A})$ is Tannakian by Theorem~\ref{thm:simple_unit_iff_tannakian}.
 
 We next show that the functor $E \colon T(\ca{B}) \rightarrow \ca{B}$ is full and faithful. First note that any object of a Tannakian category is locally free of constant finite rank, as witnessed by any fiber functor with target a category of vector spaces. From \cite[Proposition~5.5.1]{SCHAEPPI_COLIMITS} it follows that the Yoneda embedding
 \[
 Y \colon \ca{B}_{\ell f} \rightarrow T(\ca{B})_{\ell f} = T(\ca{B})
 \]
 is an equivalence of categories. By construction, the composite $E \circ Y$ is isomorphic to the inclusion $\ca{B}_{\ell f} \rightarrow \ca{B}$. Thus $\ca{E}$ is indeed full and faithful with essential image given by $\ca{B}_{\ell f}$.
 
 Finally, it remains to check that $E$ is exact. Since the domain of $E$ is Tannakian, this follows from the fact that any right exact symmetric strong monoidal functor whose domain is rigid and whose codomain is closed is also left exact (see \cite[Corollaire~2.10(i)]{DELIGNE}).
 \end{proof}


\bibliographystyle{amsalpha}
\bibliography{universal}

\providecommand{\bysame}{\leavevmode\hbox to3em{\hrulefill}\thinspace}
\providecommand{\MR}{\relax\ifhmode\unskip\space\fi MR }
\providecommand{\MRhref}[2]{%
  \href{http://www.ams.org/mathscinet-getitem?mr=#1}{#2}
}
\providecommand{\href}[2]{#2}
\begin{thebibliography}{McC42}

\bibitem[Del90]{DELIGNE}
P.~Deligne, \emph{Cat{\'e}gories tannakiennes}, The {G}rothendieck
  {F}estschrift, {V}ol.\ {II}, Progr. Math., vol.~87, Birkh{\"a}user Boston,
  Boston, MA, 1990, pp.~111--195. \MR{MR1106898 (92d:14002)}

\bibitem[Del14]{DELIGNE_SEMI_SIMPLE}
Pierre Deligne, \emph{Semi-simplicit{\'e} de produits tensoriels en
  caract{\'e}ristique {$p$}}, Invent. Math. \textbf{197} (2014), no.~3,
  587--611. \MR{3251830}

\bibitem[DS86]{DAY_STREET_GENERATORS}
Brian Day and Ross Street, \emph{Categories in which all strong generators are
  dense}, J. Pure Appl. Algebra \textbf{43} (1986), no.~3, 235--242. \MR{868984
  (88m:18013)}

\bibitem[Fla67]{FLANDERS}
Harley Flanders, \emph{Tensor and exterior powers}, J. Algebra \textbf{7}
  (1967), 1--24. \MR{0212044}

\bibitem[Kel05]{KELLY_BASIC}
G.~M. Kelly, \emph{Basic concepts of enriched category theory}, Repr. Theory
  Appl. Categ. (2005), no.~10, vi+137 pp. (electronic), Reprint of the 1982
  original [Cambridge Univ. Press, Cambridge; MR0651714]. \MR{MR2177301}

\bibitem[McC42]{MCCOY}
N.~H. McCoy, \emph{Remarks on divisors of zero}, Amer. Math. Monthly
  \textbf{49} (1942), 286--295. \MR{0006150}

\bibitem[Ram71]{RAMRAS}
Mark Ramras, \emph{Free exterior powers}, J. Algebra \textbf{19} (1971),
  110--115. \MR{0280476}

\bibitem[Sch12]{SCHAEPPI_STACKS}
Daniel Sch{\"a}ppi, \emph{A characterization of categories of coherent sheaves
  of certain algebraic stacks}, preprint,
  \href{http://arxiv.org/abs/1206.2764v1}{\tt arXiv:\ 1206.2764v1\ [math.AG]},
  2012.

\bibitem[Sch14]{SCHAEPPI_INDABELIAN}
\bysame, \emph{Ind-abelian categories and quasi-coherent sheaves}, Math. Proc.
  Cambridge Philos. Soc. \textbf{157} (2014), no.~3, 391--423. \MR{3286515}

\bibitem[Sch15]{SCHAEPPI_COLIMITS}
\bysame, \emph{Constructing colimtis by gluing vector bundles}, preprint,
  \href{https://arxiv.org/abs/1505.04596}{\tt arXiv:\ 1505.04596v1\ [math.AG]},
  2015.

\bibitem[Sch18]{SCHAEPPI_GEOMETRIC}
\bysame, \emph{Which abelian tensor categories are geometric?}, J. Reine Angew.
  Math. \textbf{734} (2018), 145--186. \MR{3739316}

\end{thebibliography}

\end{document}